\documentclass{article}
\usepackage[left=30truemm, right=30truemm]{geometry}
\usepackage{amsmath,amssymb,amsthm}
\usepackage{bm}
\usepackage{mathtools}

\theoremstyle{definition}
\newtheorem{theorem}{Theorem}[]
\newtheorem{proposition}[theorem]{Proposition}
\newtheorem{corollary}[theorem]{Corollary}

\newtheorem{remark}[theorem]{Remark}
\newtheorem{lemma}[theorem]{Lemma}
\newtheorem{claim}[]{Claim}
\newtheorem{subclaim}{Subclaim}[claim]
\newtheorem{conjecture}[theorem]{Conjecture}

\numberwithin{equation}{section}

\title{New Invariants for Partitioning a Graph into 2-connected Subgraphs}
\author{
  Michitaka Furuya\thanks{College of Liberal Arts and Sciences, Kitasato University, 1-15-1 Kitasato, Minami-ku, 
  Sagamihara 252-0373, Japan. e-mail: michitaka.furuya@gmail.com.}, 
  Masaki Kashima\thanks{School of Fundamental Science and Technology,
  Graduate School of Science and Technology, Keio University, 3-14-1 Hiyoshi, Kohoku-ku, Yokohama 223-8522, Japan. email: masaki.kashima10@gmail.com}, 
  Katsuhiro Ota\thanks{Department of Mathematics, Keio University, 3-14-1 Hiyoshi, Kohoku-ku, Yokohama 223-8522, Japan. email: ohta@math.keio.ac.jp}
}

\begin{document}
\maketitle

\begin{abstract}
Let $G$ be a graph of order $n$. 
For an integer $k\geq 2$, a partition $\mathcal{P}$ of $V(G)$ is called a $k$-proper partition of $G$ if every $P\in\mathcal{P}$ induces a $k$-connected subgraph of $G$.
This concept was introduced by Ferrara et al.~\cite{Ferrara}, and 
Borozan et al. gave minimum degree conditions for the existence of a $k$-proper partition.
In particular, when $k=2$, 
they proved that if $\delta(G)\geq\sqrt{n}$, then $G$ has a 2-proper partition $\mathcal{P}$ with
$|\mathcal{P}|\leq \frac{n-1}{\delta(G)}$.
Later, Chen et al.~\cite{Chen} extended the result by giving a minimum degree sum condition
for the existence of a 2-proper partition.
In this paper, we introduce two new invariants of graphs $\sigma^*(G)$ and $\alpha^*(G)$,
which are defined from degree sum of particular independent sets.
Our result is that if $\sigma^*(G)\geq n$, then with some exceptions, 
$G$ has a 2-proper partition $\mathcal{P}$ with $|\mathcal{P}|\leq \alpha^*(G)$.
We completely determine exceptional graphs.
This result implies both of results by Borozan et al.~\cite{Borozan} and by Chen et al.~\cite{Chen}.
Moreover, we obtain a minimum degree product condition for the existence of a 2-proper partition
as a corollary of our result.
\end{abstract}

\section{Introduction}

Throughout the paper, we only consider simple and finite graphs.
For a graph $G$, $|G|$ denotes the order of $G$.
For a graph $G$ and a vertex $u\in V(G)$, $d_G(u)$ denotes the degree of $u$ 
and $N_G(u)$ denotes the neighborhood of $u$.
In addition, we use $N_G[u]=N_G(u)\cup\{u\}$ as the closed neighborhood of $u$.
For a positive integer $k$, $[k]$ denotes the set of positive integers at most $k$.
For a positive integer $n$, $K_n$ denotes the complete graph of order $n$, 
and $K_{n,n}$ denotes the complete bipartite graph with two parts of size $n$.
For the notation and terminology not defined in this paper, 
we refer the readers to Diestel~\cite{Diestel}.

\subsection{Degree, degree sum and degree product}\label{basicinv}

For a graph $G$, $\delta(G)$ denotes the minimum degree of $G$.
When the minimum degree of $G$ is large, $G$ has many edges and 
is expected to contain many structures (such as Hamiltonian cycle). 
Thus there are many results which give a lower bound of the minimum degree for graphs to have a particular structure.
One classical extended concepts of the minimum degree is the 
\textit{minimum degree sum}, which is defined by 
$\sigma_2(G)=\min\{d_G(u)+d_G(v)\mid u,v\in V(G), u\neq v, uv\notin E(G)\}$.
When a graph $G$ is complete, we define $\sigma_2(G)=+\infty$.
If a graph $G$ satisfies $\delta(G)\geq a$, then it follows easily that $\sigma_2(G)\geq 2a$, 
so the minimum degree sum is a natural extension of the minimum degree.
Recently, Furuya and Tsuchiya~\cite{Furuya} introduced another extension of the minimum degree, the \textit{minimum  degree product}.
The minimum degree product of a graph $G$ is defined by 
$\pi_2(G)=\min\{d_G(u)d_G(v)\mid u,v\in V(G), u\neq v, uv\notin E(G)\}$.
Again, we define $\pi_2(G)=+\infty$ when $G$ is complete.
By definitions, if a graph $G$ satisfies $\delta(G)\geq a$, then $G$ satisfies $\pi_2(G)\geq a^2$.
This invariant is a natural extension of the minimum degree as well, 
but to our knowledge, there is no other result with minimum degree product condition than one by Furuya and Tsuchiya~\cite{Furuya}.
One reason for that is the following relation between the minimum degree sum and the minimum degree product;
if $\pi_2(G)\geq a^2$, then $\sigma_2(G)\geq 2a$,
which can be easily shown by the inequality of geometric and arithmetric means.
For example, as an extension of Dirac's theorem~\cite{Dirac} on hamiltonicity,
we can show that if a graph $G$ of order $n$ satisfies $\pi_2(G)\geq \frac{n^2}{4}$, 
then $G$ is hamiltonian.
However, this is an easy corollary of Ore's theorem~\cite{Ore}.
Hence, minimum degree product conditions should be considered only when
there is a gap between the minimum degree sum condition and the double of the minimum degree condition for the problem.
In this paper, we show that the minimum degree product condition works well for a problem of vertex partition.
Also, we introduce a new invariant of graphs which is a further extension of the minimum degree product.

\subsection{Partitioning a graph into highly connected subgraphs}\label{knownrslt}

For a positive integer $k$ and a graph $G$, we say $G$ is \textit{$k$-connected} if $|G|>k$ and
$G-S$ is connected for any $S\subseteq V(G)$ with $|S|<k$.
Mader~\cite{Mader} proved that if $G$ satisfies $|E(G)|/|G|\geq 2k$, then $G$ has a $k$-connected subgraph.
Motivated from this result, Ferrara et al.~\cite{Ferrara} introduced the concept of $k$-proper partition.
(The name ``$k$-proper partition'' was introduced later by Borozan et al.~\cite{Borozan}.)
Let $G$ be a graph and $\mathcal{P}$ be a partition of $V(G)$.
For an integer $k$ at least 2, we say $\mathcal{P}$ is a \textit{$k$-proper partition} of $G$
if every $P\in\mathcal{P}$ induces a $k$-connected subgraph of $G$.
Ferrara et al.~\cite{Ferrara} showed the following result which gives a minimum degree condition for the existence of a $k$-proper partition.

\begin{theorem}[\cite{Ferrara}]\label{kpp1}
  Let $G$ be a graph of order $n$, and $k$ be an integer at least 2.
  If $\delta(G)\geq 2k\sqrt{n}$, then $G$ has a $k$-proper partition $\mathcal{P}$ with $|\mathcal{P}|\leq \frac{2kn}{\delta(G)}$.
\end{theorem}

Later, this minimum degree condition was improved by Borozan et al.~\cite{Borozan}, who proved the following theorem.

\begin{theorem}[\cite{Borozan}]\label{kpp2}
  Let $G$ be a graph of order $n$, and $k$ be an integer at least 2.
  If $\delta(G)\geq \sqrt{c(k-1)n}$, then $G$ has a $k$-proper partition $\mathcal{P}$ with
  $|\mathcal{P}|\leq \frac{cn}{\delta(G)}$, where $c=\frac{2123}{180}$.
\end{theorem}

Their proof of Theorem \ref{kpp2} uses a known result of the edge density for a graph to have a $k$-connected subgraph.
They conjectured that the constant $c$ in Theorem \ref{kpp2} can be reduced to 1.

\begin{conjecture}[\cite{Borozan}]\label{kppconj}
  Let $G$ be a graph of order $n$, and $k$ be an integer at least 2.
  If $\delta(G)\geq\sqrt{(k-1)n}$, then $G$ has a $k$-proper partition $\mathcal{P}$
  with $|\mathcal{P}|\leq\frac{n-k+1}{\delta(G)-k+2}$.
\end{conjecture}

If Conjecture \ref{kppconj} has a positive answer, then
both of the lower bound of the minimum degree and the upper bound of $|\mathcal{P}|$ are sharp.
In addition, they verified Conjecture \ref{kppconj} for the case $k=2$.

\begin{theorem}[\cite{Borozan}]\label{delta}
  Let $G$ be a graph of order $n$.
  If $\delta(G)\geq \sqrt{n}$, then $G$ has a 2-proper partition $\mathcal{P}$ 
  with $|\mathcal{P}|\leq \frac{n-1}{\delta(G)}$.
\end{theorem}

Their proof of Theorem \ref{delta} gives a method of constructing a 2-proper partition from the block decomposition of $G$.
This idea is used in other results (including ours) to construct a 2-proper partition.
Chen et al.~\cite{Chen} extended Theorem \ref{delta} to the result with a minimum degree sum condition as follows.

\begin{theorem}[\cite{Chen}]\label{sigma}
  Let $G$ be a non-complete graph of order $n$ with minimum degree $\delta$.
  If either $\delta\geq \sqrt{n}$, or $1\leq\delta\leq\sqrt{n}-1$ and $\sigma_2(G)\geq\frac{n}{\delta}+\delta-1$, 
  then $G$ has a 2-proper partition $\mathcal{P}$ with
  $|\mathcal{P}|\leq \frac{2(n-1)}{\sigma_2(G)}$.
\end{theorem}

Chen et al.~\cite{Chen} did not assume the condition that $G$ is non-complete.
However, if $G$ is complete, then $\sigma_2(G)=+\infty$, and so the upper bound of $|\mathcal{P}|$ is incorrect.
Thus, to give a correct statement, we have added the condition.
Note that Theorem \ref{sigma} says nothing about a graph $G$ with $\sqrt{|G|}-1<\delta(G)<\sqrt{|G|}$.
In fact, there are infinitely many graphs $G$ with the inequality
(see the paragraph preceeding Corollary \ref{sigmarevise} in Subsection \ref{mainrslt}).
In addition, the assumption $\sigma_2(G)\geq \frac{n}{\delta}+\delta-1$ is equivalent to the inequality
$\delta(\sigma_2(G)-\delta)\geq n-\delta$, 
so it looks like an assumption on a kind of degree product.
This observation and the fact that $\sqrt{n}$ appears in the minimum degree condition in Theorem \ref{delta}
lead us to the idea of considering minimum degree product conditions for the problem of 2-proper partition.

\subsection{The minimum degree sum on large independent sets and the light independence number}\label{newinv}
To state our main result, we first introduce a new invariant $\sigma^*(G)$ as follows.
For a graph $G$, let $\mathcal{I}(G)$ denote the family of independent sets of $G$.
For each $I\in\mathcal{I}(G)$, let $\delta_G(I)=\min\{d_G(u)\mid u\in I\}$ and $w_G(I)=\sum_{u\in I}d_G(u)$.
We say $I\in\mathcal{I}(G)$ is \textit{large} if $|I|\geq\delta_G(I)+1$.
The \textit{minimum degree sum on large independent sets} is defined by
$\sigma^*(G)=\min\{w_G(I)\mid I\in\mathcal{I}(G), I \textrm{\:is\:large}\}$.
If there is no large independent set of $G$, then we define $\sigma^*(G)=+\infty$.
Our main result, which we state in the next subsection, deals with the graphs with $\sigma^*(G)\geq |G|$.
The next proposition describes the relations of the parameters $\sigma_2(G)$, $\pi_2(G)$ and $\sigma^*(G)$.

\begin{proposition}\label{thmrelation1}
  Let $G$ be a graph of order $n$ with minimum degree $\delta$.
  If $\delta\geq 1$ and $\sigma_2(G)\geq \frac{n}{\delta}+\delta-1$, then $\pi_2(G)\geq n-\delta$.
  Furthermore, if $\pi_2(G)\geq n-\delta$, then $\sigma^*(G)\geq n$.
\end{proposition}

\begin{proof}
  We first show the former statement.
  Suppose that $\delta\geq 1$ and $\sigma_2(G)\geq \frac{n}{\delta}+\delta-1$.
  We may assume that $\pi_2(G)<+\infty$.
  Let $u$ and $v$ be non-adjacent vertices of $G$ with $d_G(u)d_G(v)=\pi_2(G)$.
  We assume $d_G(u)\leq d_G(v)$ without loss of generality.
  If $d_G(u)\leq\frac{\sigma_2(G)}{2}$, then we have $\delta\leq d_G(u)\leq\frac{\sigma_2(G)}{2}$ and 
  \begin{align*}
    d_G(u)d_G(v)\geq d_G(u)(\sigma_2(G)-d_G(u))\geq\delta(\sigma_2(G)-\delta).
  \end{align*}
  Otherwise, 
  \begin{align*}
    d_G(u)d_G(v)>\left(\frac{\sigma_2(G)}{2}\right)^2\geq \delta(\sigma_2(G)-\delta).
  \end{align*}
  In either case, we have 
  $\pi_2(G)=d_G(u)d_G(v)\geq\delta(\sigma_2(G)-\delta)\geq\delta\left(\frac{n}{\delta}-1\right)=n-\delta$.

  Next, we show the latter statement.
  Suppose that $\pi_2(G)\geq n-\delta$.
  We may assume that $\sigma^*(G)<+\infty$.
  Let $I$ be a large independent set of $G$ with $w_G(I)=\sigma^*(G)$, 
  and let $u\in I$ be a vertex which attains $\delta_G(I)=d_G(u)$.
  For the moment, suppose that $d_G(u)=0$.
  Since $\sigma^*(G)\geq n>d_G(u)$, there exists a vertex $v\in I\setminus\{u\}$.
  Then $0=d_G(u)d_G(v)\geq\pi_2(G)\geq n-\delta$, and so $\delta\geq n$, 
  which is a contradiction.
  Thus $d_G(u)\geq 1$.
  As $d_G(u)d_G(v)\geq \pi_2(G)\geq n-\delta$ for each $v\in I\setminus\{u\}$, we have
  \begin{align*}
    \sigma^*(G)=w_G(I)=d_G(u)+\sum_{v\in I\setminus\{u\}}d_G(v)\geq d_G(u)+d_G(u)\frac{n-\delta}{d_G(u)}=d_G(u) +n-\delta\geq n,
  \end{align*}
  as desired.
\end{proof}

Next, we introduce another new invariant $\alpha^*(G)$.
We say $I\in\mathcal{I}(G)$ is \textit{light} if $w_G(I)\leq |G|-1$.
The \textit{light independence number} of a graph $G$ is defined by $\alpha^*(G)=\max\{|I|\mid I\in\mathcal{I}(G), I \textrm{\:is\:light}\}$.
For any graph $G$, we know that $\alpha^*(G)\geq 1$. 
By the definition, $\alpha^*(G)$ is at most the independence number $\alpha(G)$,
and the gap $\alpha(G)-\alpha^*(G)$ can be arbitrarily large.
For example, the complete bipartite graph $K_{n,n}$ satisfies $\alpha^*(K_{n,n})=1$ while $\alpha^*(K_{n,n})=n$.

The next proposition gives a relation of $\alpha^*(G)$ and the number of parts of 2-proper partition $\mathcal{P}$ in Theorem \ref{sigma}.

\begin{proposition}\label{thmrelation2}
  For any non-complete graph $G$ with $\delta(G)\geq 1$, $\alpha^*(G)\leq \frac{2(n-1)}{\sigma_2(G)}$.
\end{proposition}
  
\begin{proof}
  Assume that $G$ is non-complete and $\delta(G)\geq 1$.
  Since $\sigma_2(G)< 2(n-1)$, i.e., $\frac{2(n-1)}{\sigma_2(G)}>1$, 
  the statement holds when $\alpha^*(G)=1$.
  Hence we assume $r:=\alpha^*(G)\geq 2$.
  Let $I=\{u_0,u_1,\dots ,u_{r-1}\}$ be a light independent set of $G$.
  Then we have
  \begin{align*}
    2(n-1)\geq 2w_G(I)=\sum_{i=0}^{r-1}(d_G(u_i)+d_G(u_{i+1}))
    \geq \sum_{i=0}^{r-1}\sigma_2(G)=r\sigma_2(G)
  \end{align*}
  where $u_r=u_0$.
  Thus $\alpha^*(G)=r\leq \frac{2(n-1)}{\sigma_2(G)}$.
\end{proof}

\subsection{Main result}\label{mainrslt}

Using two invariants $\sigma^*(G)$ and $\alpha^*(G)$, we give a sufficient condition for the existence of a 2-proper partition with small number of parts.

\begin{theorem}\label{ind}
  Let $G$ be a graph of order $n$. 
  If $\sigma^*(G)\geq n$, then either $G$ has a 2-proper partition $\mathcal{P}$ with $|\mathcal{P}|\leq \alpha^*(G)$, 
  or $G$ is isomorphic to a graph in $\{K_2, F_5\}\cup\mathcal{F}_{11}\cup\mathcal{F}_{12}\cup\mathcal{H}_n$, 
  where the exceptional graphs are defined below.
\end{theorem}

\begin{figure}
  \centering
  \begin{minipage}{0.24\columnwidth}
    \centering
    \includegraphics[width=0.8\columnwidth]{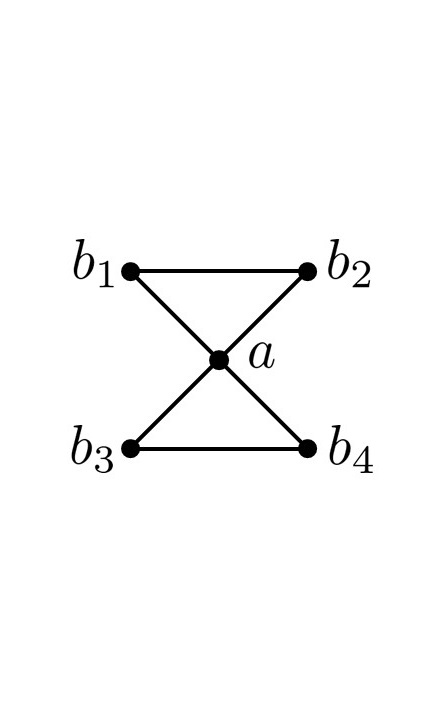}
    \caption{$F_5$}
    \label{fig:f5}
  \end{minipage}
  \begin{minipage}{0.24\columnwidth}
    \centering
    \includegraphics[width=0.8\columnwidth]{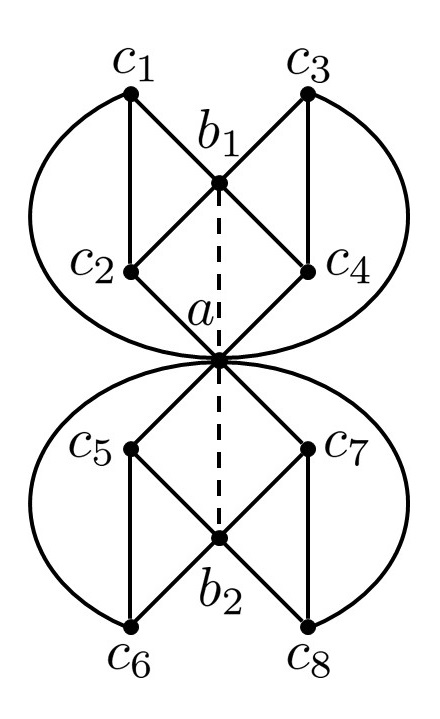}
    \caption{Graphs in $\mathcal{F}_{11}$}
    \label{fig:f11}
  \end{minipage}
  \begin{minipage}{0.24\columnwidth}
    \centering
    \includegraphics[width=0.8\columnwidth]{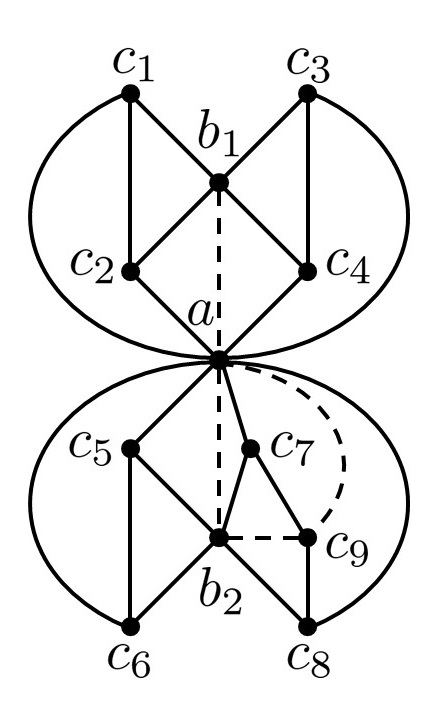}
    \caption{Graphs in $\mathcal{F}_{12}$}
    \label{fig:f12}
  \end{minipage}
  \begin{minipage}{0.24\columnwidth}
    \centering
    \includegraphics[width=0.8\columnwidth]{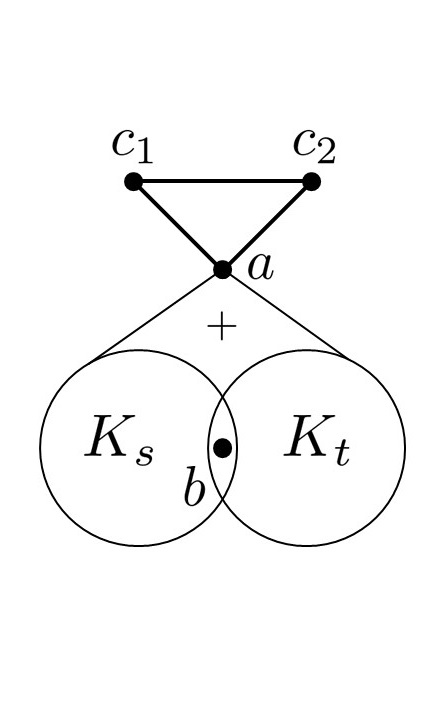}
    \caption{$H_{s,t}$}
    \label{fig:hn}
  \end{minipage}
\end{figure}

We define exceptional graphs in Theorem \ref{ind} (Figures \ref{fig:f5}-\ref{fig:hn}).
\begin{itemize}
  \item A graph $F_5$ is defined by $V(F_5)=\{a\}\cup\{b_i \mid i\in[4]\}$
  and $E(F_5)=\{ab_i\mid i\in[4]\}\cup\{b_1b_2, b_3,b_4\}$.
  \item Let $F$ be a graph having the vertex set $V(F)=\{a,b_1,b_2\}\cup\{c_i\mid i\in[8]\}$
  and the edge set $E(F)=\{ac_i\mid i\in[8]\}\cup\{b_1c_i\mid i\in[4]\}\cup\{b_2c_i\mid i\in[8]\setminus[4]\}
  \cup\{c_1c_2, c_3c_4, c_5c_6, c_7c_8\}\cup L$ where $L$ is a subset of $\{ab_1, ab_2\}$.
  Let $\mathcal{F}_{11}$ be the family of such graphs $F$.
  Note that $\mathcal{F}_{11}$ consists of three graphs up to isomorphic.
  \item Let $F$ be a graph having the vertex set $\{a, b_1, b_2\}\cup\{c_i\mid i\in[9]\}$
  and the edge set $E(F)=\{ac_i\mid i\in[8]\}\cup\{b_1c_i\mid i\in[4]\}\cup\{b_2c_i\mid i\in[8]\setminus[4]\}
  \cup\{c_1c_2, c_3c_4, c_5c_6, c_7c_9, c_8c_9\}\cup L$ 
  where $L$ is a subset of $\{ab_1, ab_2, ac_9, b_2c_9\}$ with $L\cap\{ac_9, b_2c_9\}\neq\emptyset$.
  Let $\mathcal{F}_{12}$ be the family of such graphs $F$.
  Note that $\mathcal{F}_{12}$ consists of twelve graphs.
  \item Let $s$ and $t$ be two integers with $2\leq s\leq t$.
  Let $S_1$ and $S_2$ be vertex disjoint complete graphs of order $s-1$ and $t-1$, respectively.
  Let $H_{s,t}$ be the graph obtained from $S_1$ and $S_2$ by adding four new vertices $a, b, c_1$ and $c_2$
  and the edge set $\{ax, bx\mid x\in V(S_1)\cup V(S_2)\}\cup\{ab, ac_1, ac_2, c_1c_2\}$.
  Let $H_{s,t}^-=H_{s,t}-ab$.
  Note that $|H_{s,t}|=|H_{s,t}^-|=s+t+2$.
  For an integer $n\geq 6$, let $\mathcal{H}_{n}=\{H_{s,t}, H_{s,t}^-\mid 2\leq s\leq t, s+t+2=n\}$.
\end{itemize}

Now we give some remarks on Theorem \ref{ind}.

\begin{remark}\label{reminf}
  The case $\sigma^*(G)=+\infty$ is allowed in Theorem \ref{ind}.
\end{remark}

\begin{remark}\label{remex}
  If $s$ and $t$ are integers with $3\leq s\leq t$ and $t\geq 4$, 
  then both $H_{s,t}$ and $H_{s,t}^-$ have a 2-proper partition with 3 parts
  while $\alpha^*(H_{s,t})=\alpha^*(H_{s,t}^-)=2$.
  Other exceptional graphs have no 2-proper partition.
\end{remark}

Using Propositions \ref{thmrelation1} and \ref{thmrelation2},
we can show some corollaries to Theorem \ref{ind}.
For each even integer $n\geq 6$, we define $\mathcal{H}_n^=$ by 
$\mathcal{H}_{n}^==\{H_{s,s}, H_{s,s}^-\mid 2s+2=n\}$.
Note that $\mathcal{H}_n^=$ is a subset of $\mathcal{H}_n$.
We obtain the minimum degree product condition for the existence of a 2-proper partition
as a corollary of Theorem \ref{ind}.

\begin{corollary}\label{pi}
  Let $G$ be a graph of order $n$ with minimum degree $\delta$. 
  If $\pi_2(G)\geq n-\delta$, then either $G$ has a 2-proper partition $\mathcal{P}$ with $|\mathcal{P}|\leq \alpha^*(G)$, 
  or $G\in\{K_2, F_5\}\cup\mathcal{F}_{11}\cup\mathcal{F}_{12}\cup\mathcal{H}_n^=$.
\end{corollary}

\begin{proof}
  Let $G$ be a graph of order $n$ with the minimum degree $\delta$.
  We assume that $\pi_2(G)\geq n-\delta$ and $G$ is isomorphic to no graph in $\{K_2, F_5\}\cup\mathcal{F}_{11}\cup\mathcal{F}_{12}\cup\mathcal{H}_n^=$.
  If $G$ is isomorphic to $H_{s,t}$ or $H_{s,t}^-$ for some integers $s$ and $t$ with $2\leq s<t$, 
  then we have $\delta=2$ and $\pi_2(G)=2s$, 
  which implies $\pi_2(G)=2s<s+t=n-2=n-\delta$, a contradiction.
  Hence $G$ is isomorphic to no graph in $\{K_2, F_5\}\cup\mathcal{F}_{11}\cup\mathcal{F}_{12}\cup\mathcal{H}_n$.
  By Theorem \ref{ind} and Proposition \ref{thmrelation1}, 
  $G$ has a 2-proper partition $\mathcal{P}$ with $|\mathcal{P}|\leq \alpha^*(G)$.
\end{proof}

For every even integer $n\geq 10$ and every $G\in\mathcal{H}^=_n$, 
we have $\frac{2(|G|-1)}{\sigma_2(G)}\geq 3$.
This together with Remark \ref{remex} implies that if $G$ is isomorphic to a graph in $\bigcup_{l\geq 5}\mathcal{H}_{2l}^=$,
then there exists a 2-proper partition $\mathcal{P}$ with $|\mathcal{P}|\leq \frac{2(n-1)}{\sigma_2(G)}$.
Thus, by using Propositions \ref{thmrelation1}, \ref{thmrelation2} and Corollary \ref{pi}, 
we obtain the following result that is stronger than Theorem \ref{sigma}.
(Note that if $G$ is isomorphic to a graph in $\{F_5\}\cup\mathcal{F}_{11}\cup\mathcal{F}_{12}\cup\mathcal{H}_6^=\cup\mathcal{H}_8^=$,
then $\sqrt{|G|}-1<\delta(G)<\sqrt{|G|}$.)

\begin{corollary}\label{sigmarevise}
  Let $G$ be a non-complete graph of order $n$ with minimum degree $\delta\geq 1$.
  If $\sigma_2(G)\geq\frac{n}{\delta}+\delta-1$, then either $G$ has a 2-proper partition $\mathcal{P}$ with $|\mathcal{P}|\leq \frac{2(n-1)}{\sigma_2(G)}$,
  or $G\in\{F_5\}\cup\mathcal{F}_{11}\cup\mathcal{F}_{12}\cup\mathcal{H}_6^=\cup\mathcal{H}_8^=$.
\end{corollary}

Now we introduce a relaxed concept of 2-proper partition as follows.
Let $G$ be a graph and $\mathcal{P}=\{V_1, V_2, \dots , V_r\}$ be a partition of $V(G)$.
We say $\mathcal{P}$ is an \textit{almost 2-proper partition} of $G$ 
if $\mathcal{P}$ satisfies the following two conditions.
\begin{enumerate}
  \item Either $G[V_1]$ is 2-connected, or $G[V_1]$ is isomorphic to $K_2$.
  \item For every $i\in\{2,3,\dots , r\}$, $G[V_i]$ is 2-connected.
\end{enumerate}
Every 2-proper partition of a graph $G$ is an almost 2-proper partition of $G$ as well.
Furthermore, we can easily verify that all exceptional graphs $G$ in Theorem \ref{ind}
have also almost 2-proper partitions $\mathcal{P}$ with $|\mathcal{P}|\leq\alpha^*(G)$.
(Note that $H_{s,t}$ has an almost 2-proper partition $\{\{c_1,c_2\}, V(H_{s,t})\setminus\{c_1,c_2\}\}$
for integers $s$ and $t$ with $2\leq s\leq t$.)
Thus we obtain the following result follows from Theorem \ref{ind}.

\begin{corollary}\label{almost2pp}
  Let $G$ be a graph of order $n$.
  If $\sigma^*(G)\geq n$, then $G$ has an almost 2-proper partition $\mathcal{P}$ with $|\mathcal{P}|\leq \alpha^*(G)$.
\end{corollary}

In Section 2, we give a proof of Theorem \ref{ind}, and in Section 3,
we give some graphs to show the sharpness of Theorem \ref{ind}.

\section{Proof of Theorem \ref{ind}}

\subsection{Fundamental properties of $\sigma^*$ and $\alpha^*$}\label{lems}

We can easily verify the following lemma for $\sigma^*(G)$.
Note that the same property holds for the minimum degree, the minimum degree sum and the minimum degree product as well.

\begin{lemma}\label{sigmacomp}
  Let $G$ be a graph and $G_1, G_2,\dots , G_k$ be the components of $G$.
  Then $\sigma^*(G)\leq \min\{\sigma^*(G_i)\mid i\in [k]\}$.
\end{lemma}

\begin{proof}
  Fix $i\in[k]$ and we prove $\sigma^*(G)\leq \sigma^*(G_i)$.
  If $\sigma^*(G_i)=+\infty$, then we are done.
  Otherwise, let $I$ be a large independent set of $G_i$ with $w_{G_i}(I)=\sigma^*(G_i)$.
  Then $I$ is an independent set of $G$ as well, and $I$ satisfies
  $|I|\geq \delta_{G_i}(I)+1=\delta_G(I)+1$.
  Hence $\sigma^*(G)\leq w_G(I)=w_{G_i}(I)=\sigma^*(G_i)$.
\end{proof}

For the light independence number $\alpha^*(G)$, we can easily verify the following useful properties.

\begin{lemma}\label{alphacomp}
  Let $G$ be a graph, and $G_1, G_2, \dots, G_k$ be the components of $G$.
  Then $\alpha^*(G)\geq \sum_{i=1}^k \alpha^*(G_i)$.
  As a result, if a graph $G$ has $k$ components, then $\alpha^*(G)\geq k$. 
\end{lemma}

\begin{proof}
  For each $i\in [k]$, let $I_i$ be a light independent set of $G_i$ with $|I_i|=\alpha^*(G_i)$.
  Then, $I=\bigcup_{i=1}^k I_i$ is an independent set of $G$ with
  $w_G(I)=\sum_{i=1}^k w_{G_i}(I_i)\leq \sum_{i=1}^{k} (|G_i|-1)\leq |G|-1$.
  Hence $\alpha^*(G)\geq |I|=\sum_{i=1}^k |I_i|=\sum_{i=1}^k\alpha^*(G_i)$.
\end{proof}

If $\alpha^*(G)\leq 1$, then either $G$ is complete, 
or $d_G(u)+d_G(v)\geq |G|$ for every non-adjacent pair of vertices $u,v\in V(G)$,
and hence $G$ is hamiltonian by Ore's theorem~\cite{Ore}.
Thus, the following lemma holds.

\begin{lemma}\label{alphacut}
  Let $G$ be a connected graph.
  If $G$ has a cut-vertex, then $\alpha^*(G)\geq 2$.
\end{lemma}

\subsection{Proof of Theorem \ref{ind}}\label{prf}

Let $\mathcal{F}$ be the set of exceptional graphs in Theorem \ref{ind}.
Suppose that the statement is false.
Let $G$ be a counterexample with the minimum order and let $n=|G|$.
If $n=1$, i.e., $G\simeq K_1$, then $\sigma^*(G)=0<n$, which is a contradiction.
Thus $n\geq 2$.
If $G$ has a vertex $u$ with degree at most 1, then there is a vertex $v\in V(G)\setminus N_G(u)$ since $G\not\simeq K_2$.
However $I=\{u, v\}$ is a large independent set of $G$ and 
$\sigma^*(G)\leq w_G(I)=d_G(u)+d_G(v)\leq 1+(n-2)<n$, a contradiction.
Thus the minimum degree of $G$ is at least 2.

\begin{claim}\label{conn}
  $G$ is connected.
\end{claim}

\begin{proof}
  Suppose that $G$ is disconnected and let $G_1,\dots ,G_k$ be the components of $G$. 
  First we show that $G_i$ is not isomorphic to any exceptional graph for each $i\in [k]$.
  Suppose that $G_i$ is isomorphic to a graph in $\mathcal{F}$ for some $i\in[k]$. 
  As $\delta(G)\geq 2$, $G_i$ is isomorphic to a graph in $\mathcal{F}\setminus\{K_2\}$.
  Then we can take an independent set $I_i$ of $G_i$ such that 
  $|I_i|=\delta_{G_i}(I_i)=\delta_G(I_i)$ and $w_G(I_i)=w_{G_i}(I_i)\leq |G_i|$.
  Let $u\in V(G)\setminus V(G_i)$ and $I=I_i\cup\{u\}$.
  Then $|I|=|I_i|+1=\delta_G(I_i)+1\geq\delta_G(I)+1$,
  and hence $I$ is a large independent set of $G$.
  This forces $\sigma^*(G)\leq w_G(I)=w_G(I_i)+d_G(u)\leq |G_i|+(n-|G_i|-1)=n-1$, a contradiction.
  Hence $G_i$ is not isomorphic to any graph in $\mathcal{F}$ for each $i\in[k]$.
  Since $\sigma^*(G_i)\geq \sigma^*(G)\geq n>|G_i|$ by Lemma \ref{sigmacomp}, 
  $G_i$ has a 2-proper partition $\mathcal{P}_i$ with $|\mathcal{P}_i|\leq\alpha^*(G_i)$.
  Then $\mathcal{P}=\bigcup_{i=1}^{k}\mathcal{P}_i$ is a 2-proper partition of $G$.
  By Lemma \ref{alphacomp}, $\mathcal{P}$ satisfies 
  $|\mathcal{P}|=\sum_{i=1}^{k}|\mathcal{P}_i|\leq \sum_{i=1}^{k}\alpha^*(G_i)\leq \alpha^*(G)$, 
  a contradiction. 
\end{proof}

When $G$ is 2-connected, $\{V(G)\}$ is a desired 2-proper partition.
Hence $G$ has a cut-vertex.
Let $\mathcal{B}$ be the family of blocks of $G$, 
and let $U$ be the set of cut-vertices of $G$.
The \textit{block-cut-vertex graph} $T$ of $G$ is defined by
$V(T)=\mathcal{B}\cup U$ and $E(T)=\{Bu\mid B\in\mathcal{B}, u\in U, u\in V(B)\}$.
By the definition of $T$, the graph $T$ is a tree and every leaf of $T$ belongs to $\mathcal{B}$.
A block $B\in\mathcal{B}$ which corresponds to a leaf of $T$ is an \textit{end-block} of $G$.
Since $\delta(G)\geq 2$, every end-block of $G$ has at least 3 vertices.

\begin{claim}\label{3end}
  Every end-block $B\in\mathcal{B}$ has at least 4 vertices.
\end{claim}

\begin{proof}
  Suppose that $B$ is an end-block of $G$ and $|B|=3$.
  As $\delta(G)\geq 2$, we have $B\simeq K_3$. 
  Write $V(B)=\{u_1, u_2, u_3\}$ where $u_3$ is a cut-vertex of $G$.
  If $G-V(B)$ is 2-connected, 
  then $\mathcal{P}=\{V(B), V(G)\setminus V(B)\}$ is a 2-proper partition of $G$ and 
  $|\mathcal{P}|=2\leq \alpha^*(G)$ by Lemma \ref{alphacut}, a contradiction.
  Hence $G-V(B)$ is not 2-connected.
  Since $G\not\simeq F_5$, we have $G-V(B)\not\simeq K_2$.
  Now we define two subsets $S_1$ and $S_2$ of $V(G)\backslash V(B)$ as follows.
  If $G-V(B)$ is disconnected, let $S_1$ be the vertex set of a component of $G-V(B)$ and let $S_2=V(G)\setminus(V(B)\cup S_1)$;
  if $G-V(B)$ is connected, let $x$ be a cut-vertex of $G-V(B)$, 
  and let $S_1, S_2\subseteq V(G)\setminus V(B)$ such that $G[S_1]$ and $G[S_2]$ are distinct components of $G-(V(B)\cup\{x\})$.
  Let $v_i\in S_i$ for $i\in[2]$, and define $I=\{u_1, v_1, v_2\}$.
  Then $|I|=3=d_G(u_1)+1\geq\delta_G(I)+1$, and hence $I$ is a large independent set of $G$.
  It follows that $w_G(I)\geq\sigma^*(G)\geq n$, and so $d_G(v_1)+d_G(v_2)=w_G(I)-2\geq n-2$.
  Note that $|N_{G-V(B)}[v_1]\cap N_{G-V(B)}[v_2]|\leq 1$, 
  and the equality holds if and only if $G-V(B)$ is connected and $v_1x, v_2x\in E(G)$.
  Furthermore, for each $i\in[2]$, $|N_{G-V(B)}[v_i]|\geq d_G(v_i)$ and the equality holds if and only if $v_iu_3\in E(G)$.
  Hence
  \begin{align*}
    n-2=|G-\{u_1, u_2\}|&\geq |N_{G-V(B)}[v_1]\cup N_{G-V(B)}[v_2]\cup\{u_3\}|\\
    &=|N_{G-V(B)}[v_1]|+|N_{G-V(B)}[v_2]|-|N_{G-V(B)}[v_1]\cap N_{G-V(B)}[v_2]|+1\\
    &\geq d_G(v_1)+d_G(v_2).
  \end{align*}
  This forces $d_G(v_1)+d_G(v_2)=n-2$, and all equalities hold in the above inequalities.
  As we mentioned above, $G-V(B)$ is connected, $v_1x, v_1u_3, v_2x, v_2u_3\in E(G)$.
  Since $|G-\{u_1,u_2\}|=|N_{G-V(B)}[v_1]\cup N_{G-V(B)}[v_2]\cup\{u_3\}|$,
  we have $V(G)=V(B)\cup S_1\cup S_2\cup \{x\}$ and $S_i\subseteq N_G[v_i]$ for each $i\in [2]$.
  Since $v_i\in S_i$ is arbitrary, we conclude that both $S_1\cup\{x\}$ and $S_2\cup\{x\}$ are cliques of $G$ and $S_1\cup S_2\subseteq N_G(u_3)$.
  By the symmetry of $S_1$ and $S_2$, we may assume that $|S_1|\leq |S_2|$.
  Let $s=|S_1|+1$ and $t=|S_2|+1$.
  Then $G$ is isomorphic to either $H_{s,t}$ or $H_{s,t}^-$ according to $u_3x\in E(G)$ or not.
  In particular, $G$ is isomorphic to a graph in $\mathcal{H}_n$, which is a contradiction.
\end{proof}

For each $B\in\mathcal{B}$, let $X_B=V(B)\setminus U$.

\begin{claim}\label{exception}
  Let $B\in\mathcal{B}$ with $X_B\neq\emptyset$.
  If there is a vertex $u\in V(B)\setminus X_B$ for which no block of $B-u$ contains $X_B$, then $G-V(B)$ is isomorphic to a graph in $\mathcal{F}\setminus\{K_2\}$.
\end{claim}

\begin{proof}
  Suppose that $X_B\neq\emptyset$ and no block of $B-u$ contains $X_B$ for some $u\in V(B)\setminus X_B$.
  Then, there are distinct vertices $x_1, x_2\in X_B$ such that $x_1x_2\notin E(G)$ and $|N_{B-u}[x_1]\cap N_{B-u}[x_2]|\leq 1$.
  Without loss of generality, we assume that $d_G(x_1)\leq d_G(x_2)$. 
  Considering these two vertices, we have
  \begin{align}
    |B|\geq |N_{B-u}[x_1]\cup N_{B-u}[x_2]\cup\{u\}|
    &\geq |N_{B-u}[x_1]|+|N_{B-u}[x_2]|-1+1 \notag \\
    &\geq d_G(x_1)+d_G(x_2)\geq 2d_G(x_1).\label{eq:bsize}
  \end{align}
  Let $G'=G-V(B)$.
  Then $d_{G'}(v)\geq d_G(v)-1$ for every $v\in V(G')$.
  We shall show that $\sigma^*(G')\geq |G'|$.
  Let $I'=\{u_1,u_2,\dots ,u_t\}$ be any large independent set of $G'$.
  We may assume that $\delta_{G'}(I')=d_{G'}(u_1)$.
  Let $l=\min\{d_G(u_1), d_G(x_1)\}$.
  Since $t=|I'|\geq d_{G'}(u_1)+1\geq d_G(u_1)\geq l$, 
  the set $I=\{u_1, u_2, \dots , u_l, x_1\}$ is well-defined.
  Since $|I|=l+1=\min\{d_G(u_1), d_G(x_1)\}+1\geq\delta_G(I)+1$, 
  $I$ is a large independent set of $G$, and hence $w_G(I)\geq n$.
  Consequently, it follows from (\ref{eq:bsize}) that
  \begin{align*}
    w_{G'}(I')=\sum_{i=1}^{t}d_{G'}(u_i)\geq\sum_{i=1}^{l}d_{G'}(u_i)
    \geq \sum_{i=1}^{l}(d_G(u_i)-1)
    &\geq w_G(I)-l-d_G(x_1)\\
    &\geq w_G(I)-2d_G(x_1)\\
    &\geq n-|B|
    =|G'|,
  \end{align*}
  and we conclude that $\sigma^*(G')\geq |G'|$.

  Suppose that $G'$ is isomorphic to no graph in $\mathcal{F}$.
  Then $G'$ has a 2-proper partition $\mathcal{P}'$ with $|\mathcal{P}'|\leq \alpha^*(G')$.
  Since $B$ is 2-connected, $\mathcal{P}=\mathcal{P}'\cup\{V(B)\}$ is a 2-proper partition of $G$.
  We shall show that $|\mathcal{P}|\leq \alpha^*(G)$.
  Let $J'=\{v_1, v_2, \dots , v_r\}$ be a light independent set of $G'$ with $r=\alpha^*(G')$.
  If $r\geq d_G(x_1)$, then $J_{x_1}:=\{x_1, v_1, v_2, \dots , v_{d_G(x_1)}\}$ 
  is a large independent set of $G$.
  Then by (\ref{eq:bsize}),
  \begin{align*}
    \sigma^*(G)\leq w_G(J_{x_1})=d_G(x_1)+\sum_{i=1}^{d_G(x_1)}d_G(v_i)
    &\leq d_G(x_1)+\sum_{i=1}^{d_G(x_1)}(d_{G'}(v_i)+1)\\
    &\leq d_G(x_1)+w_{G'}(J')+d_G(x_1)\\
    &\leq |G'|-1+2d_G(x_1)
    \leq n-1,
  \end{align*}
  which contradicts the assumption of the theorem.
  Thus $r\leq d_G(x_1)-1$.
  Let $J=J'\cup\{x_1\}$. 
  Again by (\ref{eq:bsize}),
  \begin{align*}
    w_G(J)=d_G(x_1)+\sum_{i=1}^{r}d_G(v_i)
    \leq d_G(x_1)+\sum_{i=1}^{r}(d_{G'}(v_i)+1)
    &=d_G(x_1)+w_{G'}(J')+r\\
    &\leq |G'|-1+2d_G(x_1)-1
    <n-1,
  \end{align*}
  and in particular, $J$ is a light independent set of $G$.
  Thus $\alpha^*(G)\geq |J|=r+1=\alpha^*(G')+1$ and hence
  $|\mathcal{P}|=|\mathcal{P}'|+1\leq \alpha^*(G')+1\leq \alpha^*(G)$, 
  which contradicts the fact that $G$ is a counterexample of the theorem.
  Therefore, $G'$ is isomorphic to a graph in $\mathcal{F}$.
  Moreover, Claim \ref{3end} implies that $G'\not\simeq K_2$, 
  and thus $G'$ is isomorphic to a graph in $\mathcal{F}\setminus\{K_2\}$.
\end{proof}

\begin{claim}\label{xblock}
  For every $B\in\mathcal{B}$ and every $u\in V(B)\setminus X_B$, $X_B$ is contained in a block of $B-u$.
\end{claim}

\begin{proof}
  Let $B_1\in\mathcal{B}$ and $u\in V(B_1)\setminus X_{B_1}$.
  Suppose that no block of $B_1-u$ contains $X_{B_1}$, and let $B_1^-=B_1-u$.
  By Claim \ref{exception}, $G-V(B_1)$ is isomorphic to a graph in $\mathcal{F}\setminus\{K_2\}$.
  Since every graph in $\mathcal{F}$ is connected, $B_1$ is an end-block of $G$, 
  and so $X_{B_1}=V(B_1^-)$.
  This together with the assumption of the proof implies that $B_1^-$ is not 2-connected.
  Now we prove the following subclaims.
  \begin{subclaim}\label{sc1}
    If $G-V(B_1)$ is isomorphic to a graph in $\{F_5, H_{2,2}, H_{2,2}^-\}$, 
    then $u$ is adjacent to all vertices of $G-V(B_1)$ having degree 2.
  \end{subclaim}
  \begin{proof}
    We may assume that $G-V(B_1)$ belongs to $\mathcal{F}\setminus\{K_2\}$.
    (For example, if $G-V(B_1)$ is isomorphic to a graph in $\mathcal{F}_{11}$, 
    then we may assume that $V(G-V(B_1))=\{a_1,b_2,b_3,c_i\mid i\in[8]\}$.)
    Since $B_1^-$ is not 2-connected, there is a vertex $v\in V(B_1^-)$ such that $d_G(v)\leq |B_1|-2$.
    
    Suppose that $G-V(B_1)\in\{F_5, H_{2,2}, H_{2,2}^-\}$.
    Then $|G-V(B_1)|\geq 5$.
    Suppose that there exists a vertex $x\in V(G)\setminus V(B_1)$ such that $d_{G-V(B_1)}(x)=2$ and $xu\notin E(G)$.
    Then we can easily verify that there exists a vertex $y\in V(G)\setminus (V(B_1)\cup\{x\})$ such that 
    $d_{G-V(B_1)}(y)=2$ and $xy\notin E(G)$.
    Let $I=\{x,y,v\}$.
    Then $I$ is a large independent set of $G$, and hence 
    $\sigma^*(G)\leq w_G(I)=d_G(x)+d_G(y)+d_G(v)\leq 2+3+(|B_1|-2)\leq
    |G-V(B_1)|+|B_1|-2=n-2$, which is a contradiction.
    Thus $u$ is adjacent to all vertices of $G-V(B_1)$ having degree 2.
  \end{proof}

  \begin{subclaim}\label{sc2}
    The graph $G-V(B_1^-)$ is a block of $G$.
  \end{subclaim}
  \begin{proof}
    We assume that $G-V(B_1)$ belongs to $\mathcal{F}\setminus\{K_2\}$, 
    and let $v$ be a vertex of $B_1^-$ such that $d_G(v)\leq |B_1|-2$.
    If $G-V(B_1)=F_5$, then it follows from Subclaim \ref{sc1} that $G-V(B_1^-)$ is a block of $G$.
    Thus we may assume that $G-V(B_1)\in\mathcal{F}\setminus\{K_2, F_5\}$.
    \begin{flushleft}
      \textbf{Case 1:} $G-V(B_1)\in\mathcal{F}_{11}\cup\mathcal{F}_{12}$.
    \end{flushleft}
    Suppose that $c_iu\notin E(G)$ for some $i\in [8]$.
    Then there exist two distinct indices $j$ and $k$ with $j, k\in [8]\setminus\{i\}$ such that 
    $\{c_i, c_j, c_k\}$ is an independent set of $G$
    (for example, if $i=1$, then $j=3$, $k=5$ are desired indices).
    Let $I=\{c_i, c_j, c_k, v\}$.
    Since $|I|=4=d_G(c_i)+1\geq \delta_G(I)+1$, 
    $I$ is a large independent set of $G$, and hence
    $\sigma^*(G)\leq w_G(I)=d_G(c_i)+d_G(c_j)+d_G(c_k)+d_G(v)\leq 3+4+4+(|B_1|-2)
    \leq |G-V(B_1)|+|B_1|-2=n-2$, which is a contradiction.
    Thus $\{c_i\mid i\in[8]\}\subseteq N_G(u)$, which implies that $G-V(B_1^-)$ is a block of $G$.
    \begin{flushleft}
      \textbf{Case 2:} $G-V(B_1)\in\mathcal{H}_l$, where $l=n-|B_1|$.
    \end{flushleft}
    Let $s$ and $t$ be integers such that $2\leq s\leq t$, $l=s+t+2$ and $G-V(B_1)\in\{H_{s,t}, H_{s,t}^-\}$.
    By the definition of $H_{s,t}$ and $H_{s,t}^-$, 
    $V(G)\setminus V(B_1)=V(S_1)\cup V(S_2)\cup\{a,b,c_1,c_2\}$.
    Suppose that $c_iu\notin E(G)$ for some $i\in[2]$.
    Let $I=\{c_i, b, v\}$.
    Since $|I|=3=d_G(c_i)+1\geq\delta_G(I)+1$, $I$ is a large independent set of $G$, 
    and hence $\sigma^*(G)\leq w_G(I)\leq 2+(s+t)+(|B_1|-2)=
    |G-V(B_1)|+|B_1|-2=n-2$, which is a contradiction.
    Thus $\{c_1, c_2\}\subseteq N_G(u)$.

    Suppose that $(V(S_1)\cup V(S_2))\cap N_G(u)=\emptyset$.
    Let $I'=\{c_1,x_1,x_2,v\}$, where $x_j$ is a vertex in $V(S_j)$ for each $j\in[2]$.
    Since $|I'|=4=d_G(c_1)+1\geq\delta_G(I')+1$, $I'$ is a large independent set of $G$, and hence
    $\sigma^*(G)\leq w_G(I')=d_G(c_1)+d_G(x_1)+d_G(x_2)+d_G(v)\leq 3+s+t+(|B_1|-2)
    =(|G-V(B_1)|+1)+|B_1|-2=n-1$, which is a contradiction.
    Thus $(V(S_1)\cup V(S_2))\cap N_G(u)\neq\emptyset$, 
    which implies that $G-V(B_1^-)$ is a block of $G$.
  \end{proof}

  Let $B_2=G-V(B_1^-)\in\mathcal{B}$ and $B_2^-=B_2-u$ ($=G-V(B_1)$).
  Since $B_2^-$ is isomorphic to a graph in $\mathcal{F}\setminus\{K_2\}$, 
  $B_2^-$ is not 2-connected.
  It follows from Subclaim \ref{sc2} that $B_2$ is an end-block of $G$,
  and hence, $X_{B_2}=V(B_2)\setminus\{u\}$.
  This implies that no block of $B_2-u$ contains $X_{B_2}$.
  Applying Claim \ref{exception} with $B=B_2$, we have $G-V(B_2)$ ($=B_1^-$) is isomorphic to a graph in $\mathcal{F}\setminus\{K_2\}$.
  In particular, the roles of $B_1$ and $B_2$ are symmetric.
  Furthermore, by exchanging the role of $B_1$ and $B_2$, the following subclaim holds.

  \begin{subclaim}\label{sc3}
    If $G-V(B_2)$ ($=B_1^-$) is isomorphic to a graph in $\{F_5, H_{2,2}, H_{2,2}^-\}$,
    then $u$ is adjacent to all vertices of $G-V(B_2)$ having degree 2.
  \end{subclaim}

  We may assume that $|B_1^-|\leq |B_2^-|$.

  \begin{subclaim}\label{sc4}
    If $B_1^-$ is isomorphic to $F_5$, then $B_2^-$ is isomorphic to a graph in $\mathcal{F}\setminus\{K_2, F_5, H_{2,2}, H_{2,2}^-\}$.
  \end{subclaim}
  \begin{proof}
    If $B_1^-$ is isomorphic to $F_5$ and $B_2^-$ is isomorphic to a graph in $\{F_5, H_{2,2}, H_{2,2}^-\}$,
    then by Subclaims \ref{sc1} and \ref{sc3}, $G$ is isomorphic to a graph in $\mathcal{F}_{11}\cup\mathcal{F}_{12}$,
    which is a contradiction.
  \end{proof}

  \begin{subclaim}\label{sc5}
    If $B_1^-$ is isomorphic to a graph in $\{H_{2,2}, H_{2,2}^-\}$,
    then $B_2^-$ is isomorphic to a graph in $\mathcal{F}\setminus\{K_2, F_5, H_{2,2}, H_{2,2}^-\}$.
  \end{subclaim}
  \begin{proof}
    Assume that $B_1^-$ is isomorphic to a graph in $\{H_{2,2}, H_{2,2}^-\}$.
    Since $6=|B_1^-|\leq |B_2^-|$, $B_2^-\not\simeq F_5$.
    Suppose that $B_2^-$ is isomorphic to a graph in $\{H_{2,2}, H_{2,2}^-\}$.
    Then for $i\in[2]$, there exist two non-adjacent vertices $x_1^{(i)}$ and $x_2^{(i)}$ of $B_i^-$
    with $d_{B_i^-}(x_1^{(i)})=d_{B_i^-}(x_2^{(i)})=2$.
    Let $I=\{x_j^{(i)}\mid i\in[2], j\in[2]\}$.
    Note that $d_G(x_j^{(i)})\leq d_{B_i^-}(x_j^{(i)})+1=3$.
    Then $I$ is a large independent set of $G$,
    and hence $\sigma^*(G)\leq w_G(I)=\sum_{i\in[2], j\in[2]}d_G(x_j^{(i)})\leq 4\cdot 3<13=n$, 
    which is a contradiction.
  \end{proof}

  Since $|B_1^-|\leq |B_2^-|$, it follows from Subclaims \ref{sc4} and \ref{sc5} that
  $B_2^-$ is isomorphic to a graph in $\mathcal{F}\setminus\{K_2, F_5, H_{2,2}, H_{2,2}^-\}$.
  Now we define an independent set $I_2$ of $B_2^-$ as follows.
  If $B_2^-$ is isomorphic to a graph in $\mathcal{F}_{11}\cup\mathcal{F}_{12}$, 
  let $I_2$ be the set of vertices of $B_2^-$ corresponding to $c_1$ and $c_3$;
  if $B_2^-$ is isomorphic to a graph in $\{H_{s,t}, H_{s,t}^-\}$ for some $s$ and $t$ with $2\leq s\leq t$,
  let $I_2$ be the set of vertices of $B_2^-$ corresponding to $c_1$ and a vertex in $V(S_1)$.
  In the former case, we have $w_{B_2^-}(I_2)=6<|B_2^-|-3$;
  in the latter case, since $t\geq 3$, $w_{B_2^-}(I_2)=s+2\leq s+t-1=|B_2^-|-3$.
  In either case, we obtain $w_{B_2^-}(I_2)\leq |B_2^-|-3$.

  Next we define the independent set $I_1$ of $B_1^-$ as follows.
  If $B_1^-\simeq F_5$, let $I_1$ be the set of vertices corresponding to $b_1$ and $b_3$;
  if $B_1^-$ is isomorphic to a graph in $\mathcal{F}_{11}\cup\mathcal{F}_{12}$, let $I_1$ be the vertices corresponding to $c_1$, $c_3$ and $c_5$;
  if $B_1^-$ is isomorphic to a graph in $\{H_{s,t}, H_{s,t}^-\}$ for some $s$ and $t$ with $2\leq s\leq t$,
  let $I_1$ be the vertices corresponding to $b$ and $c_1$.
  Then we have $|I_1|=\delta_{B_1^-}(I_1)$ and $w_{B_1^-}(I_1)\leq |B_1^-|-|I_1|+1$.

  Let $\tilde{I}=I_1\cup I_2$.
  Since $|\tilde{I}|=|I_1|+2\geq (\delta_{B_1^-}(I_1)+1)+1\geq \delta_{G}(I_1)+1\geq \delta_G(\tilde{I})+1$, 
  $\tilde{I}$ is a large independent set of $G$, and hence
  \begin{align*}
    \sigma^*(G)\leq w_G(\tilde{I})=w_G(I_1)+w_G(I_2)&\leq w_{B_1^-}(I_1)+|I_1|+w_{B_2^-}(I_2)+|I_2|\\
    &\leq (|B_1^-|-|I_1|+1)+|I_1|+(|B_2^-|-3)+2=n-1,
  \end{align*}
  which is a contradiction.
\end{proof}

Let $X=\bigcup_{B\in\mathcal{B}}X_B$.

\begin{claim}\label{xadjacent1}
  For every $u\in V(G)$, $N_G(u)\cap X\neq \emptyset$.
  In particular, for every $B\in\mathcal{B}$, if $X_B\neq \emptyset$, then $|X_B|\geq 2$.
\end{claim}
\begin{proof}
  Let $u\in V(G)$.
  Suppose that $N_G(u)\cap X=\emptyset$.
  Write $N_G(u)=\{v_1, v_2,\dots , v_r\}$.
  For each $i\in [r]$, let $D_i$ be a component of $G-v_i$ which does not contain $u$.
  Then
  \begin{align}
    V(D_i)\cap V(D_j)=\emptyset \textrm{\: for\: all\:} i,j\in [r] \textrm{\: with\:} i\neq j\label{eq:disjoint}
  \end{align}
  Let $y_i$ be a vertex of $D_i$ for each $i\in[r]$.
  Then we have $|D_i|\geq d_G(y_i)$.
  Let $I=\{u, y_1, y_2, \dots , y_r\}$.
  The condition (\ref{eq:disjoint}) implies $I$ is a large independent set of $G$, 
  and hence $w_G(I)\geq n$.
  Again, by (\ref{eq:disjoint}), we have
  \begin{align*}
    |G|\geq |N_G[u]|+\sum_{i=1}^r |D_i|\geq (d_G(u)+1)+\sum_{i=1}^r d_G(y_i)=w_G(I)+1\geq n+1,
  \end{align*}
  which is a contradiction.
\end{proof}

Fix an end-block $B_0$ of $G$.
We regard the block-cut-vertex graph $T$ of $G$ as a rooted tree with the root $B_0$.
For each $B\in\mathcal{B}\setminus\{B_0\}$, let $u_B$ be the parent of $B$ in $T$.
By Claim \ref{xblock}, for each $B\in\mathcal{B}\setminus\{B_0\}$ with $X_B\neq\emptyset$, there is a block $A_B$ of $B-u_B$ that contains $X_B$.

\begin{claim}\label{xadjacent2}
  Let $B\in\mathcal{B}\setminus\{B_0\}$ with $X_B\neq\emptyset$.
  If $u\in V(B)\setminus(V(A_B)\cup\{u_B\})$, then $N_G(u)\cap(X\setminus X_B)\neq\emptyset$.
\end{claim}

\begin{proof}
  Let $B$ be a block in $\mathcal{B}\setminus\{B_0\}$ and let $u\in V(B)\setminus(V(A_B)\cup\{u_B\})$.
  Write $N_G(u)=\{v_1, v_2, \dots, v_r\}$.
  Suppose that $N_G(u)\cap(X\setminus X_B)=\emptyset$.
  By Claim \ref{xadjacent1}, we have $N_G(u)\cap X_B\neq\emptyset$.
  Since $u\notin V(A_B)$ and $X_B\subseteq V(A_B)$, we have $|N_G(u)\cap X_B|=1$.
  Without loss of generality, we may assume that $v_r\in X_B$ and $v_i\notin X$ for all $i\in[r-1]$.
  By Claim \ref{xadjacent1}, there is a vertex $x\in X_B\setminus\{v_r\}$.
  Since $u\notin V(A_B)$, we have 
  \begin{align}
    |N_G(u)\cap N_G(x)|\leq |\{u_B, v_r\}|=2.\label{eq:overlap}
  \end{align}
  Fix $i\in[r-1]$.
  Let $D_i$ be a component of $G-v_i$ which does not contain $u$.
  Since $B$ is a block of $G$, $B-v_i$ is a connected graph containing $u$ if $v_i\in V(B)$.
  This leads to $V(D_i)\cap V(B)=\emptyset$.
  Let $y_i\in V(D_i)$ be a vertex.
  Then we have
  \begin{align}
    |D_i|\geq |N_G[y_i]\setminus\{v_i\}|\geq d_G(y_i).\label{eq:compsize}
  \end{align}
  Let $I=\{u, x, y_1, y_2, \dots , y_{r-1}\}$.
  Then $I$ is a large independent set of $G$, and hence $w_G(I)\geq n$.
  Recall that $V(D_i)\cap V(B)=\emptyset$ for all $i\in[r-1]$.
  Hence by (\ref{eq:overlap}) and (\ref{eq:compsize}), we have
  \begin{align*}
    n=|G|\geq |N_G[u]\cup N_G[x]|+\sum_{i=1}^{r-1}|D_i|
    &\geq |N_G[u]|+|N_G[x]|-2+\sum_{i=1}^{r-1}d_G(y_i)\\
    &=d_G(u)+d_G(x)+\sum_{i=1}^{r-1}d_G(y_i)\\
    &=w_G(I)\geq n.
  \end{align*}
  This together with (\ref{eq:compsize}) forces $y_iv_i\in E(G)$ and $V(D_i)=N_G[y_i]\setminus\{v_i\}$ for each $i\in[r-1]$.
  Since $y_i\in V(D_i)$ is arbitrary, $V(D_i)\cup\{v_i\}$ is a clique of $G$.
  Furthermore, $V(G)$ is the disjoint union of $V(B)$, $N_G(u)\setminus V(B)$ and $\bigcup_{i=1}^{r-1}V(D_i)$.
  For each $i\in[r-1]$, we define $V_i$ as $V_i=V(D_i)\cup\{v_i\}$ if $v_i\notin V(B)$,
  and $V_i=V(D_i)$ otherwise.
  By Claim \ref{3end}, $|V_i|\geq 3$ for each $i\in[r-1]$.
  In addition, since $w_G(I\setminus\{x\})<w_G(I)=n$, 
  $I\setminus\{x\}$ is a light independent set of $G$, and hence $\alpha^*(G)\geq |I\setminus\{x\}|=r$.
  Combining these, we obtain a 2-proper partition $\mathcal{P}:=\{V(B), V_2, V_3,\dots , V_r\}$ of $G$
  with $|\mathcal{P}|=r\leq\alpha^*(G)$, a contradiction.
\end{proof}

\begin{claim}\label{asize}
  For every $B\in\mathcal{B}\setminus\{B_0\}$ with $X_B\neq\emptyset$, $A_B$ has at least 3 vertices.
  In particular, $A_B$ is 2-connected.
\end{claim}

\begin{proof}
  If $B\in\mathcal{B}\setminus\{B_0\}$ is an end-block of $G$, then $|A_B|=|B-u_B|\geq 3$ by Claims \ref{3end} and \ref{xblock}.
  Suppose that $B\in\mathcal{B}\setminus\{B_0\}$ is not an end-block, $X_B\neq\emptyset$ and $|A_B|\leq 2$.
  By Claim \ref{xadjacent1}, we know that $V(A_B)=X_B$ and $|A_B|=2$.
  Write $X_B=\{x_1, x_2\}$. 
  Since $A_B$ is a connected graph, we have $x_1x_2\in E(G)$.
  Since $B$ is not an end-block of $G$, one of $x_1$ and $x_2$, say $x_1$, is adjacent to a vertex $v_1\in V(B)\setminus(X_B\cup\{u_B\})$.
  Write $N_G(x_2)=\{x_1, v_2, v_3, \dots , v_r\}$ where $r=d_G(x_2)$.
  Since $v_1\notin V(A_B)$, $v_1$ is not adjacent to $x_2$.
  For each $i\in[r]$, let $D_i$ be a component of $G-v_i$ which does not contain $x_2$,
  and let $y_i\in V(D_i)$.
  Then we have $|D_i|\geq d_G(y_i)$.
  Since $V(D_i)\cap V(B)=\emptyset$ for every $i\in[r]$ and $V(D_i)\cap V(D_j)=\emptyset$ for all distinct $i,j\in [r]$, 
  $I=\{x_2, y_1, y_2,\dots , y_r\}$ is a large independent set of $G$.
  Hence 
  \begin{align*}
    |G|\geq |N_G[x_2]|+\sum_{i=1}^{r}|D_i|\geq (d_G(x_2)+1)+\sum_{i=1}^{r}d_G(y_i)=w_G(I)+1\geq n+1,
  \end{align*} 
  which is a contradiction.
\end{proof}

For each $B\in\mathcal{B}$, let $G(B)$ be a subgraph of $G$ induced by the vertices of $B$ and its descendant blocks with respect to $T$.
In addition, for each $B\in\mathcal{B}$, let $\mathcal{B}(B)=\{B'\in\mathcal{B}\mid B'\subseteq G(B)\}$
and let $\tilde{\mathcal{B}}(B)=\{B'\in\mathcal{B}(B)\mid X_{B'}\neq\emptyset\}$.

\begin{claim}\label{2pp}
  Let $B\in\mathcal{B}\setminus\{B_0\}$.
  Then $G(B)-u_B$ has a 2-proper partition $\mathcal{P}_B^-$ with $|\mathcal{P}_B^-|=|\tilde{\mathcal{B}}(B)|$. 
  Furthermore, if $X_B\neq\emptyset$, then $G(B)$ has a 2-proper partition $\mathcal{P}_B$ with $|\mathcal{P}_B|=|\tilde{\mathcal{B}}(B)|$.
\end{claim}

\begin{proof}
  The proof goes by induction on the height $h$ of the block-cut-vertex graph of $G(B)$ with the root $B$. 
  When $h=0$, $B$ is an end-block of $G$, so $G(B)=B$, $X_B=V(B)\setminus\{u_B\}\neq\emptyset$ and $|\tilde{\mathcal{B}}(B)|=1$.
  By Claim \ref{xblock}, $\{X_{B}\}$ is a desired 2-proper partition of $G(B)-u_B$, 
  and $\{V(B)\}$ is a desired 2-proper partition of $G(B)$.
  Hence we assume $h\geq 1$, i.e., $B$ has a child in $T$.
  For each $x\in V(B)\setminus(X_B\cup\{u_B\})$, let $\mathcal{C}(x)$ be the family of blocks in $\mathcal{B}\setminus\{B\}$ containing $x$,
  and let $\mathcal{C}(B)=\bigcup_{x\in V(B)\setminus(X_B\cup\{u_B\})}\mathcal{C}(x)$.

  By the induction hypothesis, the statements of the claim hold for every $B'\in\mathcal{C}(B)$.
  For each block $B'\in \mathcal{C}(B)$, let $\mathcal{P}^-_{B'}$ be a 2-proper partition of $G(B')-u_{B'}$ with $|\mathcal{P}^-_{B'}|=|\tilde{\mathcal{B}}(B')|$.
  In addition, for each block $B'\in\mathcal{C}(B)$ with $X_{B'}\neq\emptyset$, 
  let $\mathcal{P}_{B'}$ be a 2-proper partition of $G(B')$ with $|\mathcal{P}_{B'}|=|\tilde{\mathcal{B}}(B')|$.

  First we assume $X_B=\emptyset$.
  For each $x\in V(B)\setminus\{u_B\}$, there is a block $B_x\in\mathcal{C}(x)$ with $X_{B_x}\neq \emptyset$ by Claim \ref{xadjacent1}. 
  Then, 
  \begin{align*}
    \mathcal{P}_B^-=\bigcup_{x\in V(B)\setminus\{u_B\}}\left(\mathcal{P}_{B_x}\cup\bigcup_{B'\in\mathcal{C}(x)\setminus\{B_x\}}\mathcal{P}^-_{B'}\right)
  \end{align*}
  is a 2-proper partition of $G(B)-u_B$ with 
  $|\mathcal{P}_B^-|=\sum_{B'\in\mathcal{C}(B)}|\tilde{\mathcal{B}}(B')|=|\tilde{\mathcal{B}}(B)|$.

  Next we assume that $X_B\neq \emptyset$. 
  Then $B$ is 2-connected, and hence
  \begin{align*}
    \mathcal{P}_B=\{V(B)\}\cup\left(\bigcup_{B'\in\mathcal{C}(B)}\mathcal{P}^-_{B'}\right)
  \end{align*}
  is a 2-proper partition of $G(B)$ with 
  $|\mathcal{P}_B|=1+\sum_{B'\in\mathcal{C}(B)}|\tilde{\mathcal{B}}(B')|=|\tilde{\mathcal{B}}(B)|$,
  which proves the second statement of the claim.
  For each $x\in V(B)\setminus (V(A_B)\cup\{u_B\})$, there is a block $B_x\in\mathcal{C}(x)$ with $X_{B_x}\neq \emptyset$ by Claim \ref{xadjacent2}. 
  By Claim \ref{asize}, $A_B$ is 2-connected, and hence
  \begin{align*}
    \mathcal{P}_B^-=\{V(A_B)\}\cup\left(\bigcup_{x\in V(B)\setminus(V(A_B)\cup\{u_B\})}\mathcal{P}_{B_x}\right)\cup
    \left(\bigcup_{B'\in\mathcal{C}(B)\setminus\{B_x\mid x\in V(B)\setminus(V(A_B)\cup\{u_B\})\}}\mathcal{P}^-_{B'}\right)
  \end{align*}
  is a 2-proper partition of $G(B)-u_B$ with 
  $|\mathcal{P}^-_B|=1+\sum_{B'\in\mathcal{C}(B)}|\tilde{\mathcal{B}}(B')|=|\tilde{\mathcal{B}}(B)|$, 
  which proves the first statement of the claim.
\end{proof}

\begin{claim}\label{fbsize}
  For every $B\in\mathcal{B}$, $|\tilde{\mathcal{B}}(B)|\leq \alpha^*(G(B))$.
\end{claim}
\begin{proof}
  Let $B\in\mathcal{B}$. 
  For each $B'\in\tilde{\mathcal{B}}(B)$, take a vertex $x_{B'}\in X_{B'}$.
  Let $I=\{x_{B'}\mid B'\in\tilde{\mathcal{B}}(B)\}$.
  Then
  \begin{align*}
    w_{G(B)}(I)=\sum_{B'\in\tilde{\mathcal{B}}(B)}d_G(x_{B'})
    \leq \sum_{B'\in\tilde{\mathcal{B}}(B)}(|B'|-1)
    \leq \sum_{B'\in\mathcal{B}(B)}(|B'|-1)=|G(B)|-1.
  \end{align*}
  This together with the definition of $X_{B'}$ implies that $I$ is a light independent set of $G(B)$.
  Consequently, $\alpha^*(G(B))\geq |I|=|\tilde{\mathcal{B}}(B)|$.
\end{proof}

Recall that $B_0$ is an end-block of $G$.
Let $u_0$ be the unique cut-vertex of $G$ belonging to $B_0$, and let $\mathcal{B}_0$ be the family of blocks in $\mathcal{B}\setminus\{B_0\}$ of $G$ containing $u_0$.
For each $B\in\mathcal{B}_0$, let $\mathcal{P}_{B}^-$ be as in Claim \ref{2pp}.
Then $\mathcal{P}_0:=\{V(B_0)\}\cup(\bigcup_{B\in\mathcal{B}_0}\mathcal{P}_B^-)$ is a 2-proper partition of $G$.
Since $B_0$ is an end-block of $G$, $X_{B_0}\neq \emptyset$.
This together with Claim \ref{fbsize} leads to $|\mathcal{P}_0|=1+\sum_{B\in\mathcal{B}_0}|\tilde{\mathcal{B}}(B)|=|\tilde{\mathcal{B}}(B_0)|\leq \alpha^*(G(B_0))=\alpha^*(G)$,
which is a contradiction.
This completes the proof of Theorem \ref{ind}.

\section{Sharpness of Theorem \ref{ind}}

In this section, we construct graphs to show the sharpness of Theorem \ref{ind}.
For positive integers $n$ and $d$ with $d(d+1)+1<n$, 
let $\bm{t}=(t_1,t_2,\dots ,t_d)$ be a tuple of integers such that 
$\sum_{i=1}^{d}t_i=n-1$ and $\min\{t_i\mid i\in[d]\}\geq d+1$.
For each $i\in[d]$, let $H_i$ be the complete graph of order $t_i$, 
and let $v_i$ be a vertex of $H_i$.
Take a new vertex $u$.
Now we define two graphs $G_{\bm{t}}$ and $G_{\bm{t}}'$ as follows.
A graph $G_{\bm{t}}$ is obtained from the disjoint union of $H_1, H_2,\dots ,H_d$ 
by adding $u$ and edges $uv_i$ for all $i\in [d]$.
A graph $G_{\bm{t}}'$ is obtained from the disjoint union of $H_1, H_2,\dots ,H_d$
by adding $u$ and all edges between $u$ and $\bigcup_{i=1}^{d}V(H_i)$.
By constructions, both $G_{\bm{t}}$ and $G_{\bm{t}}'$ have order $n$.

We show that $G_{\bm{t}}$ satisfies $\sigma^*(G_{\bm{t}})=n-1$ and has no 2-proper partition.
Let $I$ be any large independent set of $G_{\bm{t}}$.
As $\delta(G_{\bm{t}})=d$, $I$ must have at least $d+1$ vertices, and thus 
$I=\{u,w_1,w_2,\dots ,w_d\}$ where $w_i\in V(H_i)\setminus\{v_i\}$ for each $i\in [d]$.
Thus we have $w_{G_{\bm{t}}}(I)=d_{G_{\bm{t}}}(u)+\sum_{i=1}^{d}d_{G_{\bm{t}}}(w_i)=d+\sum_{i=1}^{d}(t_i-1)=n-1$.
Since $I$ is arbitrary, it follows that $\sigma^*(G_{\bm{t}})=n-1$.
Since no 2-connected subgraph of $G_{\bm{t}}$ contains $u$, $G_{\bm{t}}$ does not have a 2-proper partition.
Hence $G_{\bm{t}}$ shows the sharpness of the lower bound of $\sigma^*(G)$ in Theorem \ref{ind}.

Next, we show that $\sigma^*(G_{\bm{t}}')=+\infty$, $\alpha^*(G_{\bm{t}}')\leq d$ and 
every 2-proper partition of $G_{\bm{t}}'$ has at least $d$ parts.
Since $\delta(G_{\bm{t}}')\geq d$ and $\alpha(G_{\bm{t}}')=d$, 
we have $\sigma^*(G_{\bm{t}}')=+\infty$.
We also have $\alpha^*(G_{\bm{t}}')\leq \alpha(G_{\bm{t}}')=d$.
Let $\mathcal{P}$ be a 2-proper partition of $G_{\bm{t}}'$.
For all distinct $i,j\in[d]$, $V(H_i)$ and $V(H_j)$ 
must be in distinct parts in $\mathcal{P}$.
Thus $\mathcal{P}$ has at least $d$ parts.
Hence $G_{\bm{t}}'$ shows the sharpness of the upper bound of $|\mathcal{P}|$ in Theorem \ref{ind}.

\section*{Acknowledgement}
This work was partially suported by JSPS KAKENHI Grant number JP23K03204 (to M.F),
JSPS KAKENHI Grant number JP22K03404 (to K.O),
Keio University SPRING scholarship Grant number JPMJSP2123 (to M.K), 
and JST ERATO Grant Number JPMJER2301 (to M.K).


\begin{thebibliography}{9}
  \bibitem{Borozan}
	V. Borozan, M. Ferrara, S. Fujita, M. Furuya, Y. Manoussakis, Narayanan N, and D. Stolee,
	Partitioning a graph into a highly connected subgraphs, 
	\textit{J. Graph Theory} 82 (2016), no. 3, 322-333.
	\bibitem{Chen}
	X. Chen, X. Guo, and X. Yang, 2-proper partition of a graph, 
	\textit{Graphs Combin.} 38 (2022), Paper No. 191, 11pp.
  \bibitem{Diestel}
  R. Diestel, Graph Theory. 5th ed., \textit{Springer}, 2016.
	\bibitem{Dirac}
	G. A. Dirac, Some theorems on abstract graphs, 
	\textit{Proc. Lond. Math. Soc.} 2 (1952), 69-81.
	\bibitem{Ferrara}
	M. Ferrara, C. Magnant, and P. Wenger, Condition for families of disjoint $k$-connected subgraphs in a graph,
	\textit{Discrete Math.} 313 (2013), 760-764.
	\bibitem{Furuya}
	M. Furuya and S. Tsuchiya, New strategy on the existence of a spanning tree without small degree stems, 
	arXiv:2303.03762.
	\bibitem{Mader}
	W. Mader, Existenz $n$-fach zusammenh\"{a}ngender Teilgraphen in Graphen gen\"{u}gend grosser Kantendichte,
	\textit{Abh. Math. Sem. Univ. Hamburg.} 37 (1972), 86-97.
	\bibitem{Ore}
	O. Ore, Note of Hamilton circuits, 
	\textit{Amer. Math. Monthly} 67 (1960), 67.
\end{thebibliography}
\end{document}